\documentclass[11pt,letterpaper]{amsart}

\usepackage[utf8]{inputenc}
\usepackage[english]{babel}
\usepackage{amsmath,amssymb,amsthm,mathrsfs}
\usepackage[colorlinks=true,linkcolor=blue,citecolor=blue,urlcolor=blue]{hyperref}
\usepackage{comment}

\newtheorem{thm}{Theorem}[section]

\newtheorem{lem}[thm]{Lemma}

\theoremstyle{definition}
\newtheorem{defn}[thm]{Definition}

\newtheorem{conj}[thm]{Conjecture}
\theoremstyle{remark}
\newtheorem{rmk}[thm]{Remark}

\newcommand{\tr}{\operatorname{tr}}
\newcommand{\diver}{\operatorname{div}}

\newcommand{\mP}{m_{\mathbf P}}

\newcommand{\Pten}{\mathbf P}

\newcommand{\kten}{\mathbf k}

\title[Penrose conjecture under $2$-convexity condition]
{The Penrose conjecture for initial data sets satisfying a $2$-convexity condition}
\author{Conghan Dong}
\address{Department of Mathematics, Duke University, 120 Science Dr, Durham, NC 27710, USA}
\email{conghan.dong@duke.edu}
\date{\today}

\begin{document}

\begin{abstract}
	Let $(M^3, g, \mathbf{k})$ be a smooth, connected, asymptotically flat initial data set with connected outermost past apparent horizon $\Sigma $. We prove the Penrose conjecture, namely that $m_{\mathrm{ADM}}(g) \geq \sqrt{\frac{|\Sigma |}{16 \pi }} $, under the assumptions of the dominant energy condition and the $2$-convexity condition that the sum of the two smallest eigenvalues of $\kten$ is nonnegative. The main tool is the $\sigma $-inverse mean curvature flow, together with a monotonicity formula developed in \cite{Dong26SigmaIMCF}.
\end{abstract}

\maketitle 

\section{Introduction and main results}

Let \((M^3,g,\kten)\) be a smooth, connected, asymptotically flat initial data set.  We use the constraint
conventions
\begin{align}\label{eq:constraints}
  16\pi\mu
  &=R+(\tr_g\kten)^2-|\kten|^2,
  &
  8\pi J
  &=\diver_g\big(\kten-(\tr_g\kten)g\big),
\end{align}
where \(R\) is the scalar curvature of \(g\).  The dominant energy condition
(DEC) is
\begin{align}\label{eq:DEC}
  \mu\geq |J|_g.
\end{align}

A past apparent horizon satisfies $H_{\Sigma } = \mathrm{tr}_{\Sigma } \mathbf{k}$, whereas a future apparent horizon satisfies $H_{\Sigma } = - \mathrm{tr}_{\Sigma } \kten$. Here $H_{\Sigma }$ is the mean curvature of $\Sigma $ in $(M^3, g)$, and $\mathrm{tr}_{\Sigma } \kten$ is the trace of $\kten$ restricted to $\Sigma $. Reversing the time orientation by replacing $\kten$ with $- \kten$ turns
 a past apparent horizon into a future apparent horizon. A generalized apparent horizon in the sense of
\cite{BrayKhuri11} satisfies
 $H_{\Sigma }=|\tr_\Sigma\kten|.$

A standard version of the Penrose conjecture, in terms of past apparent horizons, can be stated as follows.

\begin{conj}
	Suppose that $(M^3, g, \kten)$ is complete, connected, asymptotically flat initial data set satisfying the dominant energy condition (\ref{eq:DEC}), and let $m$ be its ADM mass. If $\Sigma $ is a past apparent horizon, then
	\begin{align*}
		m \geq \sqrt{\frac{A}{16 \pi }} ,
	\end{align*} 
	where $A$ is the area of the outermost minimal area enclosure of $\Sigma $. 
\end{conj}

When $\mathbf{k}=0$, the Penrose conjecture reduces to the Riemannian Penrose inequality, which was proved by Huisken-Ilmanen \cite{HI01} in the case of a connected horizon and by Bray \cite{Bray01} in full generality. Beyond this setting, several special cases are also known. In particular, the spherically symmetric case has been established; see \cite{Hayward96, IMM96, BrayKhuri10Jang}. For other  recent results towards the general Penrose conjecture, see \cite{AllenBrydenKazarasKhuriPenrose, AnHeKerrPenrose, EllithyPenrose} and the reference therein. 

The Penrose conjecture for another large class of initial data sets, in which $\kten$ is proportional to $g$, was recently proved in \cite{Dong26SigmaIMCF}. In that work, the author introduced a general $\sigma $-inverse mean curvature flow and derived a new monotonicity formula. In this paper, we further apply those techniques to study the initial data sets satisfying the $2$-convexity condition.

More precisely, if $\kappa_3 \geq  \kappa_2 \geq \kappa_1$ are the eigenvalues of $\kten$ with respect to $g$, then the $2$-convexity condition of $\kten$ is given by
$$\kappa_1 + \kappa_2 \geq 0.$$
In particular, it's the intermediate case between $\kten \geq 0$ and $\mathrm{tr}_g \kten \geq 0$.

Define the tensor
\begin{align}\label{eq:def-P}
  \Pten :=(\tr_g\kten)g-\kten.
\end{align}
The $2$-convexity condition of $\kten$ is equivalent to 
$$\Pten \geq 0 .$$

Since
\begin{align*}
	|\Pten|^2 = |\kten|^2 + (\mathrm{tr}_g \kten)^2,\ \mathrm{tr}_g\Pten = 2 \mathrm{tr}_g \kten,
\end{align*} 
in terms of $\Pten$, the DEC (\ref{eq:DEC}) becomes
\begin{align}\label{eq:DEC2}
	R \geq |\Pten|^2 - \frac{1}{2} (\mathrm{tr}_g \Pten)^2 + 2 |\mathrm{div}_g \Pten|.
\end{align}

In this paper, we will be interested in the case in which $\Pten \geq 0$ globally as a tensor. If $\Pten \leq 0$,  then one can reduce to the nonnegative case by considering the past apparent horizon of the initial data set $(M^3, g, - \mathbf{k})$, which is the same as the future apparent horizon of original initial data set $(M^3, g, \kten)$. 

For every two-sided surface \(\Sigma\) with unit normal $\nu $,
\begin{align}\label{eq:P-basic}
  \Pten (\nu,\nu)=\tr_\Sigma\kten.
\end{align}
If $\Pten \geq 0$, then
\(\tr_\Sigma\kten\geq 0\), and the past apparent horizon equation coincides with the generalized apparent horizon equation, namely \(H=\Pten(\nu,\nu)\).

Our main result is the following:
\begin{thm}\label{thm-main}
	Suppose that $(M^3, g, \kten)$ is complete, connected, asymptotically flat initial data set satisfying the dominant energy condition (\ref{eq:DEC}), and let $m$ be its ADM mass. Assume the $2$-convexity condition that $(\mathrm{tr}_{g} \kten) g - \kten \geq 0$. If $\Sigma = \partial M $ is a connected outermost past apparent horizon, then
	\begin{align*}
		m \geq \sqrt{\frac{\mathrm{Area}(\Sigma )}{16 \pi }} .
	\end{align*} 
	Equality holds if and only if $\kten \equiv 0$ and $(M^3, g)$ is isometric to the standard Schwarzschild manifold.
\end{thm}

\begin{rmk}
	The ADM mass $m$ is defined in (\ref{defn-mass}),  which is called the ADM energy in the literature. Under the asymptotical flat assumption (\ref{AF-condition-0}), the $2$-convexity assumption will force zero ADM linear momentum, so that the ADM energy equals the mass.
\end{rmk}

\begin{rmk}
	The class of initial data sets $(M^3, g, \kten)$ satisfying DEC and $2$-convexity condition contains a large class of nontrivial examples. The intersection of this class $\{\kten: (\mathrm{tr}_{g} \kten) g - \kten \geq 0\} $ and the proportional class $\{ \kten: \kten = \frac{\tau }{3} g\} $ is given by $\{\kten: \kten = \frac{\tau }{3} g, \tau  \geq 0\} $, which already contains a large class of nontrivial examples satisfying DEC.
 See \cite[Examples 2.1]{Dong26SigmaIMCF}.
\end{rmk}
\begin{rmk}
	Under the $2$-convexity condition, the standard Penrose conjecture is equivalent to the generalized Penrose conjecture. If one assumes the boundary of $M$ consists of past trapped surfaces, i.e. $H_{\Sigma } \leq \mathrm{tr}_{\Sigma } \kten$, then the existence of the unique $C^{2,\alpha }$-smooth outermost past apparent horizon is given by \cite{Eichmair10}.
\end{rmk}

The main technique used to prove the theorem is the $\sigma $-inverse mean curvature flow developed in \cite{Dong26SigmaIMCF}, with $\sigma = \Pten$:
\begin{align}\label{eq:P-IMCF}
  \frac{\partial F}{\partial t}=\frac{\nu}{H - \Pten (\nu , \nu )},
\end{align}
where $F: N^2 \times [0, T] \to M^3$ is a family of hypersurfaces $N_t := F(N, t)$. 
 For a smooth closed leaf
 \(N_t\),  we follow \cite{Dong26SigmaIMCF} and define the quantities
\begin{align}\label{eq:mass-quantities}
  A(t)
  &:=e^{-t}|N_t|,
  \\
  B(t)
  &:=e^{t/2}\left(
     1-\frac{1}{16\pi}\int_{N_t}\left( H- \Pten (\nu ,\nu ) \right) ^2
  \right),
  \\
  \mP(N_t)
  &:=\sqrt{\frac{A(t)}{16\pi}}\,B(t)
   =\sqrt{\frac{|N_t|}{16\pi}}
     \left(1-\frac{1}{16\pi}
     \int_{N_t}\left( H - \Pten(\nu ,\nu ) \right) ^2 \right).
\end{align}
Under the DEC (\ref{eq:DEC}) and the assumption that $\Pten \geq 0$, one shows that $A(t)$ and $ B(t)$ are monotone nondecreasing along smooth solutions of (\ref{eq:P-IMCF}) as long as $N_t$ remains connected. Hence $\mP(N_t)$ is monotone nondecreasing whenever $B(t) \geq 0$. Furthermore, one proves that $\lim_{t\to \infty} m_{\Pten}(N_t) \leq m $.

The weak theory of the flow (\ref{eq:P-IMCF}) has been studied in \cite{Dong26SigmaIMCF}. A related weak existence theory for inverse null mean curvature flow were developed in \cite{Moore12}, in which the flow has a different form $\frac{\partial F}{\partial t} = \frac{\nu }{H+ \Pten(\nu ,\nu )}$ under the assumption $\mathrm{tr}_{g} \kten \geq 0$. See also \cite{HuiskenWolff22} for a related inverse space-time mean curvature flow in the maximal case $\mathrm{tr}_{g} \kten =0$. Thus, it remains to extend the monotonicity formula to weak solutions. This follows from modifications of \cite{Dong26SigmaIMCF}.  The main result follows by starting the weak flow from the connected outermost past apparent horizon.  

We organize this paper as follows. In Section \ref{section-preliminary}, we recall preliminary results from \cite{Dong26SigmaIMCF} and indicate the necessary modifications. In Section \ref{section-monotone}, we prove the monotonicity formula along smooth solutions. The remainder of the proof of the main theorem then follows the argument of \cite{Dong26SigmaIMCF}. 

\subsection*{Acknowledgement} I would like to thank Hubert Bray for helpful discussions and for suggesting the convexity condition. I also thank Marcus Khuri for helpful comments.

\section{Preliminaries}\label{section-preliminary}
\subsection{Definition of initial data sets}
An initial data set for a spacetime consists of a triple $(M^3, g, \mathbf{k})$, where $M$ is a three-dimensional manifold, $g$ is a Riemannian metric, and $\mathbf{k}$ is a symmetric $(0,2)$-tensor. 

Following \cite{SchoenYau81}, an initial data set $(M, g, \mathbf{k})$ is called asymptotically flat if, outside a compact set, $M$ decomposes into finitely many ends $M_1, \ldots, M_{p}$, each diffeomorphic to the complement of a compact subset of $\mathbb{R}^{3}$, and under such diffeomorphisms, the tensors $g$ and $\mathbf{k}$ satisfy the decay conditions
\begin{align}\label{AF-condition-0}
	\begin{split}
		&|g_{ij} - \delta _{ij}| + |x| \cdot |\partial g_{ij}| + |x|^2 \cdot |\partial ^2 g_{ij}| \leq C |x| ^{-1},\\
	& |R| + |x| \cdot |\partial R| \leq C |x|^{-4},\\
	& |\mathbf{k}_{ij}| + |x|\cdot  |\partial \mathbf{k}_{ij}|  + |x|\cdot  \left| \textstyle \sum_{i} \mathbf{k}_{ii}\right|  \leq C |x|^{-2},
	\end{split}
\end{align} 
as $|x|\to \infty$, where norms and derivatives are taken with respect to the Euclidean metric. 

With each end $M_k$ we associate the ADM mass $m_k$ defined by the flux integral
\begin{align}\label{defn-mass}
	m_k = \frac{1}{16 \pi } \lim_{r\to \infty}   \int_{S_r} \sum_{i,j}^{}\left( g_{ij,j} - g_{jj,i} \right) n ^{i} dA,
\end{align} 
where $S_r= \{x: |x|=r\} $ is the Euclidean $2$-sphere at infinity. 

For convenience, we modify the topology of $M$ by compactifying all of the ends of $M$ except for one chosen end. We denote by $m= m(g)$ the ADM mass of the chosen end.

\subsection{Classical solution of $\sigma $-IMCF}
Let $\sigma $ be a smooth nonnegative symmetric $2$-tensor.  Let $F: N^{n-1} \times [0, T] \to M^{n}$ be a family of hypersurfaces $N_t:= F(N, t)$. Following \cite{Dong26SigmaIMCF}, we say $F$ is a classical solution of the \textit{$\sigma $-inverse mean curvature flow} (or $\sigma $-IMCF) if the following parabolic evolution equation holds in the classical sense:
\begin{align}\label{*-eq}
	\frac{\partial F}{\partial t}(x,t) = \frac{\nu }{H- |\nu |^2_{\sigma }}(x,t),\quad x \in N,\ 0 \leq t \leq T, \tag*{$(*)$}
\end{align} 
where $H$ is the mean curvature of $N_t$ at the point $F(x, t)$, $\nu $ is the outward unit normal, $|\nu |^2_{\sigma }= \sigma _{ij} \nu ^{i} \nu ^{j} \geq 0$, $H-|\nu |^2_{\sigma }$ is assumed to be positive, $\frac{\partial F}{\partial t}$ denotes the normal velocity along the surface $N_t$, and all derivatives and geometric quantities are well defined. 

By \cite[Section 3]{Dong26SigmaIMCF}, we have
\begin{lem}\label{evolution-eq-classical}
Consider the classical solutions of \ref{*-eq}. For any smooth cutoff function $\phi $ so that $\mathrm{supp} \phi \cap \partial N_{t} = \emptyset$, we have
\begin{align}\label{ddt-phi-area}
	\begin{split}
		\frac{d}{d t} \int_{N_t} \phi &= \int_{N_t} \frac{\nabla _{\nu } \phi + \phi H}{H-|\nu |^2_{\sigma }},
	\end{split}
\end{align}
and
 \begin{align}\label{dt-int-H-nu}
	\begin{split}
		&\frac{d}{d t} \int_{N_t} \phi (H- |\nu |^2_{\sigma })^2 \\
									&= \int_{N_t} (H-|\nu |^2_{\sigma }) \nabla _{\nu }\phi -2 \langle \nabla ^{N} \phi , \frac{\nabla ^{N}(H-|\nu |^2_{\sigma })}{H- |\nu |^2_{\sigma }} \rangle \\
									&\ \ - 2 \int_{N_t}\phi \left(   \frac{|\nabla ^{N} (H-|\nu |^2_{\sigma })|^2}{(H-|\nu |^2_{\sigma })^2} + \left( \mathrm{Ric}(\nu ,\nu ) +|\mathrm{II}|^2 \right) \right) \\
									&\ \ -2 \int_{N_t} \phi (\nabla _{\nu } \sigma) (\nu , \nu ) - 4 \int_{N_t} \phi \frac{\sigma ( \nu , \nabla ^{N}(H- |\nu |^2_{\sigma }) )}{H- |\nu |^2_{\sigma }} + \int_{N_t}\phi H(H-|\nu |^2_{\sigma }).
	\end{split}
\end{align}

\end{lem}

\subsection{Weak solution of $\sigma $-IMCF}
The level-set description of the evolution by $\sigma $-IMCF can be formulated as follows. We assume that the evolving surfaces are given by the level-sets of a function $u: M \to \mathbb{R}$ via
\begin{align*}
	E_t:= \{ x: u(x) <t\} ,\quad N_t := \partial E_t. 
\end{align*}
When $u$ is smooth with $\nabla u \neq 0$, equation \ref{*-eq} is equivalent to 
\begin{align}\label{*2-eq}
	\mathrm{div}\left( \frac{\nabla u}{|\nabla u|} \right) = |\nabla u| + \frac{\sigma _{ij} \nabla ^{i} u \nabla ^{j}u}{|\nabla u|^2}, \tag*{$(* *)$}
\end{align}
where the left side is the mean curvature of $\{ u = t\} $ and $|\nabla u|$ is the inverse speed.

Following \cite[Section 4]{Dong26SigmaIMCF}, by freezing the $|\nabla u| + \sigma (\nu ,\nu )$ term on the right-hand side of \ref{*2-eq}, for compact subset $K \subset M$, we consider equation \ref{*2-eq} as the Euler-Lagrange equation of the functional
\begin{align*}
	J_{u, \nu }^{\sigma }(v) = J_{u, \nu }^{\sigma ,K}(v) := \int_{K}\left( |\nabla v| + v (|\nabla u| + \sigma (\nu ,\nu )) \right).
\end{align*} 
Alternatively, we define the functional
\begin{align*}
	J_{u, \nu }^{\sigma }(F) = J_{u, \nu }^{\sigma ,K}(F):= |\partial ^* F \cap K| - \int_{F \cap K} \left( |\nabla u| + \sigma (\nu ,\nu ) \right) ,
\end{align*} 
for sets $F$ of locally finite perimeter.

\begin{defn}\label{defn-weak}
	Let $E_0 \subset M$ be a precompact open set with $C^2$-boundary $N_0= \partial E_0$, and $\sigma $ be a Lipschitz nonnegative symmetric $2$-tensor. We call the pair $(U, \tilde{\nu })$ a \textit{weak solution} of \ref{*2-eq} with initial condition $E_0$ if $U \in C^{0,1}_{loc}(M \times \mathbb{R})$ and $\tilde{\nu } $ is a measurable unit vector field which satisfy
	\begin{itemize}
		\item [(i)] $U$ is translation invariant in the vertical direction in the sense that $U(x,z) = u(x)$ for a locally Lipschitz function $u: M \to \mathbb{R}$ satisfying $u \geq 0$ on $M \setminus E_0$, $u|_{\partial E_0} =0$, $u<0$ in $E_0$, and $u(x) \to \infty$ as $d(x, E_0) \to \infty$.
		\item [(ii)] $\tilde{E}_t := \{ U< t\} $ minimizes $J_{U, \tilde{\nu }} ^{\sigma }$ in $(M\setminus \bar{E}_0) \times \mathbb{R}$ for each $t>0$, and $\sigma $ the translation invariant extension. At each jump time $t_0$, each point $\tilde{x}_0=(x_0, z_0)$ in the interior $\tilde{\mathcal{K}}_{t_0}$ of the jump region $\{U = t_0\} $ lies in the boundary $\partial \tilde{E}_{\tilde{x}_0} \in C^{1,\alpha }_{loc}$ of a Caccioppoli set $\tilde{E}_{\tilde{x}_0}$ that minimizes $J_{U, \tilde{\nu }}$ in $\tilde{\mathcal{K}}_{t_0}$.
		\item [(iii)] $\tilde{\nu }$ is a translation invariant in the sense that $\tilde{\nu }(x, z_1) = \tilde{\nu} (x, z_2)$ for any $z_1, z_2 \in \mathbb{R}$, $\tilde{\nu }(\tilde{x}) \in C^{\alpha }_{loc}$ away from jump times and is the unit normal vector to $\partial \tilde{E}_t$ at each point $\tilde{x} \in \partial \tilde{E}_t$, and $\tilde{\nu }(\tilde{x}) \in C^{1,\alpha }_{loc}(\tilde{\mathcal{K}}_{t_0}) $ is the unit normal vector to $\partial \tilde{E}_{\tilde{x}_0}$ at each point $\tilde{x} \in \partial \tilde{E}_{\tilde{x}_0}$ and each jump time $t_0$.
	\end{itemize}
\end{defn}

To obtain the existence of weak solutions, one consider the $\epsilon $-translating graphs as follows.

For any $L \gg 1$, by \cite[Lemma 4.1]{Dong26SigmaIMCF}, there exists a function $u ^{\epsilon }$, where $\epsilon = \epsilon (L) \to 0$ as $L\to \infty$ so that for
\begin{align}
	U ^{\epsilon }(x,z):= u ^{\epsilon }(x) - \epsilon z,\quad (x,z) \in \Omega _{L} \times \mathbb{R},
\end{align}
the following equation holds:
\begin{align}
	\tilde{\mathcal{E}}(U^{\epsilon }):= \mathrm{div}_{\tilde{g}}\left( \frac{\tilde{\nabla } U^{\epsilon }}{|\tilde{\nabla } U^{\epsilon }|} \right) - |\tilde{\nabla }U ^{\epsilon }| - \frac{\sigma _{ij} \tilde{\nabla }^{i} U^{\epsilon } \tilde{\nabla }^{j} U^{\epsilon }}{|\tilde{\nabla }U^{\epsilon }|^2} = 0 \text{ in } \Omega _{L} \times \mathbb{R},
\end{align}
where $\tilde{g}= g + dz ^2$ is the product metric, $\sigma (x,z) = \sigma (x)$ is the constant extension along the $z$-direction, and $\Omega _{L} \to M^3$ as $L\to \infty$. So $U^{\epsilon }$ is a smooth solution of \ref{*2-eq} in $\Omega _{L} \times \mathbb{R}$.  For any $-\infty< t< \infty$, set
\begin{align*}
	\tilde{N}^{\epsilon }_t := \{ U^{\epsilon }= t\} = \mathrm{graph} \left( \frac{u^{\epsilon }}{\epsilon } - \frac{t}{\epsilon } \right) .
\end{align*} 
Then $\tilde{N}^{\epsilon }_{t}$ is a smooth solution of the $\sigma $-IMCF \ref{*-eq} in $\Omega _{L} \times \mathbb{R}$:
\begin{align*}
	\frac{\partial }{\partial t} \tilde{F} = \frac{\tilde{\nu }_{\epsilon }}{\tilde{H} - |\tilde{\nu }_{\epsilon }|^2_{\sigma }},
\end{align*}
where $\tilde{\nu }_{\epsilon } = \frac{\tilde{\nabla }U^{\epsilon }}{|\tilde{\nabla }U^{\epsilon }|}$ is the smooth downward unit normal vector. 
Notice that $\tilde{N}^{\epsilon }_{t}$ has boundary consisting of translations of $\partial \Omega _{L}$. 

Using the local uniform Lipschitz estimates together with the Arzela-Ascoli theorem, and by regularity results of geometric measure theory, one can show that there exists $L_i \to \infty, \epsilon _i \to 0$ so that
\begin{align*}
	U ^{\epsilon _i}(x,z) \to U(x,z) = u(x)
\end{align*} 
locally uniformly, $\tilde{\nu }_{\epsilon _i} \to \tilde{\nu }$ a.e., and the limit $(U, \tilde{\nu })$ provides a weak solution of the $\sigma $-IMCF. More precisely, we have the following existence theorem.

\begin{thm}[{\cite[Theorem 4.9]{Dong26SigmaIMCF}}]\label{sgm-IMCF-existence}
	Let $(M^n, g)$ be a complete, connected Riemannian $n$-manifold without boundary, and let $\sigma $ be a smooth nonnegative symmetric $2$-tensor with a uniform global $C^{0,1}$-bound. Suppose there exists a proper smooth subsolution of \ref{*2-eq} with a precompact initial condition. Then for any nonempty, precompact, open set $E_0 \subset M^n$ with $C^2$-boundary, there exists a weak solution of the $\sigma $-IMCF with initial condition $E_0$ as in Definition \ref{defn-weak}.
\end{thm}
As a corollary, we obtain
\begin{thm}\label{P-IMCF-existence}
	Let $(M^3, g, \kten)$ be a complete, connected, asymptotically flat initial data set without boundary. Assume that $\Pten \geq 0$. Then for any nonempty, precompact, open set $E_0 \subset M$ with $C^2$-boundary, there exists a weak solution $(U, \tilde{\nu })$ of the $\sigma$-IMCF with $\sigma = \Pten$ and with initial condition $E_0$ as in Definition \ref{defn-weak}.
\end{thm}
In the following, we denote by $(u, \nu )= (U, \tilde{\nu }|_{TM})$, and we will simply say that $(u, \nu )$ is a weak solution of the $\Pten$-IMCF. We write $E_t:= \{ u< t\} $, $E_{t}^{+}:= \{ u \leq t\} $, $N_t := \partial E_t$ and $N_t ^{+}: = \partial E_t ^{+}$.

\subsection{Outward optimizing hulls}
We slightly modify the definition of outward optimizing hull and notice that the same arguments and results still work. We assume $\Pten \geq 0$ in the following.

Let $\Omega $ be an open set. We call $E$ an \textit{outward optimizing hull with respect to $\nu $} (with weight $\Pten $ in $\Omega $) if $E$ minimizes area minus bulk energy $\Pten (\nu ,\nu )$ on the outside in $\Omega $, that is, if
\begin{align}\label{defn-optimizing-hull}
	|\partial ^* E \cap K| - \int_{E \cap K} \Pten (\nu ,\nu ) \leq |\partial ^* F \cap K| - \int_{F \cap K} \Pten (\nu ,\nu ),
\end{align}
for any $F$ containing $E$ such that $F \setminus E \subset \subset \Omega $, and any compact set $K$ containing $F \setminus E$. We say that $E$ is \textit{strictly outward optimizing with respect to $\nu $} (with weight $\Pten$ in $\Omega $) if equality implies that $F \cap \Omega = E \cap \Omega $ a.e. 

Let $E$ be any measurable set. Define $E' = E'_{\Omega }$ to be the intersection of (the Lebesgue points of) all the strictly outward optimizing hulls in $\Omega $ that contain $E$. Up to sets of measure zero, this may be realized by a countable intersection, so $E'$ itself is a strictly outward optimizing hull and open. Due to the asymptotic flatness of $g$ and $\kten$, if $E \subset \subset M$, then $E'$, taken in $M$, exists and is precompact as well. See also \cite{Eichmair10}. In particular, if $M \setminus E_0$ is an exterior region, since $\partial E_0$ is outermost generalized apparent horizon, we know $E_0$ is strictly outward optimizing.
If $\partial E$ is $C^2$, then we know $\partial E'$ is $C^{2,\alpha }$ wherever it does not contact the obstacle $E$. 

\begin{lem}[{\cite[Lemma 5.2]{Dong26SigmaIMCF}}]\label{lem-optimizing-property}
	Suppose that $(u, \nu )$ is a weak solution of the $\Pten$-IMCF with initial condition $E_0$. Then
	\begin{itemize}
		\item [(i)] $E_t$ is outward optimizing in $M$ for $t>0$;
		\item [(ii)] $E_t ^{+}$ is outward optimizing in $M$ for $t \geq 0$;
		\item [(iii)] $|\partial E_t ^{+}| = |\partial E_t| + \int_{E_t ^{+}\setminus E_t}\Pten(\nu ,\nu )$, for all $t>0$. This extends to $t=0$ precisely if $E_0$ is outward optimizing.
	\end{itemize}
\end{lem}

By the outward optimizing property of $E_{t}^{+}$ for each $t \geq 0$ and the $C^{1,\alpha }$-regularity of $\partial E ^{+}_{t}$, we know that in the weak sense
\begin{align}\label{mean-curv-hull}
	\begin{split}
		H - \Pten(\nu ,\nu ) &= 0 \text{ on } \partial E_{t} ^{+} \setminus \partial E_{t},\\
		H - \Pten(\nu ,\nu ) & \geq 0 \quad \mathcal{H}^{n-1}\text{-a.e. on } \partial E_{t}^{+} \cap \partial E_{t}.
	\end{split}
\end{align}
By the standard elliptic regularity, $\partial E_{t}^{+} \setminus \partial E_{t}$ is in fact $C^{2,\alpha }$-smooth so that the mean curvature is also well defined in the classical sense. In particular, if we start a weak flow from the outermost generalized apparent horizon $E_0$ as constructed in \cite{Eichmair10}, then $E_{0} ^{+} = E_0$.

\subsection{Uniqueness and smoothness of weak solutions}
The same argument of \cite[Theorem 5.6]{Dong26SigmaIMCF} gives the following
\begin{thm}\label{thm-smooth-uniqueness}
	Let $E_0$ be a precompact open set in $M$ such that $\partial E_0$ is $C^2$ with $H- \Pten(\nu ,\nu )>0$ and $E_0 = E_0 ^{+}$. Then there exists a weak solution $(E_t)_{0<t<\infty}$ of the $\Pten$-IMCF with initial condition $E_0$, as a limiting of $\epsilon _i$-translating graphs, coincides with the $C^{2,\alpha }$-classical solution for a short time, provided that $E_t$ remains precompact for a short time. In particular, $u$ is $C^{2,\alpha }$ for a short time.
\end{thm}

\subsection{Topological consequences}
Let $(M^3, g, \kten)$ be a complete, connected, asymptotically flat initial data set with $\Pten \geq 0$. We allow $M$ to have a compact $C^2$-boundary consisting of past trapped surfaces, that is, surfaces $\Sigma $ satisfying
\begin{align*}
	H_{\Sigma } \leq \mathrm{tr}_{\Sigma }\kten.
\end{align*} 
Notice that since $\mathrm{tr}_{\Sigma } \kten \geq 0$ under our assumptions, minimal surfaces and traditional apparent horizons are always generalized trapped surfaces. 

As in \cite[Section 5]{Dong26SigmaIMCF}, we have the following topological consequences. For any fixed end, define the exterior region $M'$ so that the interior of $M'$ includes the given end and contains no compact immersed generalized apparent horizons. By \cite[Theorem 1.1]{Eichmair10} and \cite[Theorem 5.1]{EGP13}, we know that $\partial M'$ encloses $\Sigma $, $\partial M'$ consists of closed embedded $C^{2,\alpha }$-smooth outermost generalized apparent horizon, which are also strictly area outerminimizing in the sense that every other surface which encloses it in $M'$ has larger area, and  $M'$ is diffeomorphic to $\mathbb{R}^3$ minus a finite number of open $3$-balls. Finally, we have the following connectedness lemma.
\begin{lem}[{\cite[Lemma 5.9]{Dong26SigmaIMCF}}]
			Suppose that $M$ is connected and simply connected with no boundary and a single, asymptotically flat end, and $(E_t)_{t>0}$ is a weak solution of the $\Pten $-IMCF with initial condition $E_0$. If $\partial E_0$ is connected, then $N_t$ remains connected as long as it stays compact.
\end{lem}

\section{Monotonicity formula and the proof of the theorem}
\label{section-monotone}
Let's first assume that $(N_{t})_{0 \leq t \leq T}\subset M$ is a smooth solution of the $\Pten$-IMCF \ref{*-eq} consisting of closed connected surfaces.

For 
\begin{align*}
	A(t) = e ^{-t}|N_t|,
\end{align*} 
by Lemma \ref{evolution-eq-classical}, 
\begin{align*}
	A'(t) &= e ^{-t} \left( \frac{d}{d t}|N_t| - |N_t| \right)\\
	      &= e ^{-t} \left( \int_{N_t} \frac{H}{H - \Pten(\nu ,\nu )} - |N_t| \right) \\
	      &= e ^{-t} \int_{N_t} \frac{\Pten(\nu ,\nu )}{H - \Pten(\nu ,\nu )}\\
	      & \geq 0.
\end{align*} 

For
\begin{align*}
	B(t) = e ^{\frac{t}{2}} \left( 1- \frac{1}{16 \pi } \int_{N_t} (H - \Pten(\nu ,\nu ))^2 \right) ,
\end{align*}
by Lemma \ref{evolution-eq-classical}, 
\begin{align*}
	\begin{split}
	16 \pi & e ^{-\frac{t}{2}}B'(t)\\
		&= 8 \pi  - \frac{1}{2} \int_{N_t}(H- \Pten(\nu ,\nu ) )^2 -\frac{d}{d t} \int_{N_t} (H- |\nu |^2_{\Pten })^2 \\
									&= 8 \pi  - \frac{1}{2} \int_{N_t}(H- \Pten(\nu ,\nu ) )^2 + 2 \int_{N_t} \left(   \frac{|\nabla ^{N} (H-|\nu |^2_{\Pten })|^2}{(H-|\nu |^2_{\Pten })^2} + \left( \mathrm{Ric}(\nu ,\nu ) +|\mathrm{II}|^2 \right) \right) \\
									&\ \ +2 \int_{N_t}  (\nabla _{\nu } \Pten ) (\nu , \nu ) + 4 \int_{N_t} \frac{\Pten ( \nu , \nabla ^{N}(H- |\nu |^2_{\Pten }) )}{H- |\nu |^2_{\Pten }} - \int_{N_t}H(H-|\nu |^2_{\Pten }).
	\end{split}
\end{align*}  
By the Gauss equation
	\begin{align*}
		R + |\mathrm{II}|^2 + H^2 = R_{N_t} + 2 \left( \mathrm{Ric}(\nu , \nu ) + |\mathrm{II}|^2 \right) ,
	\end{align*} 
	and using the trace free tensor $\mathring{\mathrm{II}} = \mathrm{II} - \frac{1}{2} H g_{N_t}$ and the Gauss-Bonnet theorem,
	we have
	\begin{align*}
		\begin{split}
			16 \pi & e ^{-\frac{t}{2}}B'(t) \\
			       &\geq  \int_{N_t} 2\left| \frac{\nabla ^{N} (H-|\nu |^2_{\Pten })}{H-|\nu |^2_{\Pten }} + \Pten ^{TN}(\nu , \cdot ) \right| ^2  - 2|\Pten ^{TN}(\nu ,\cdot )|^2  \\
			       & + 8 \pi - 4 \pi \chi (N_t) - \frac{1}{2} \int_{N_t} (H- \Pten(\nu ,\nu ) )^2 + \int_{N_t} \left( R + |\mathring{\mathrm{II}}|^2 + \frac{3}{2}H^2 \right) \\
			       &+2 \int_{N_t}  (\nabla _{\nu } \Pten ) (\nu , \nu ) - \int_{N_t}H(H-\Pten (\nu ,\nu ))\\
			       &\geq  \int_{N_t} 2\left| \frac{\nabla ^{N} (H-|\nu |^2_{\Pten })}{H-|\nu |^2_{\Pten }} + \Pten ^{TN}(\nu , \cdot ) \right| ^2 + |\mathring{\mathrm{II}}|^2 \\
			       & + \int_{N_t} \left( R - \frac{1}{2} (\Pten(\nu ,\nu  ) )^2 + 2 H \Pten(\nu ,\nu ) - 2 |\Pten ^{TN}(\nu ,\cdot )|^2 + 2 (\nabla _{\nu }\Pten)(\nu ,\nu ) \right) ,
		\end{split}
	\end{align*} 
	where $\Pten ^{TN}(\nu , \cdot )$ denotes the restriction of $\Pten(\nu ,\cdot )$ to the tangent plane of $N_t$. 
Note that
\begin{align}\label{nabla-P-id}
	\begin{split}
		(&\nabla _{\nu } \Pten) (\nu ,\nu ) \\
		&= (\mathrm{div}_{g} \Pten) (\nu ) - \mathrm{div}_{N_t}(\Pten ^{TN}(\nu ,\cdot ) ) -H \mathbf{P}(\nu ,\nu ) + \langle \mathrm{II}, \mathbf{P}^{TN} \rangle \\
&= (\mathrm{div}_{g} \Pten) (\nu ) - \mathrm{div}_{N_t}(\Pten ^{TN}(\nu ,\cdot ) )
						   + \frac{1}{2} H\mathrm{tr}_g \Pten - \frac{3}{2}H \Pten(\nu ,\nu )   + \langle \mathring{\Pten} ^{TN}, \mathring{\mathrm{II}} \rangle,
	\end{split}
\end{align} 
where $\mathring{\Pten} ^{TN} = \Pten ^{TN} - \frac{1}{2} \mathrm{tr}_{N_t} \Pten ^{TN} g_{N_t}	$ is the trace free part.
So
\begin{align*}
\begin{split}
	&16 \pi  e ^{-\frac{t}{2}}B'(t) \\
	&\geq  \int_{N_t} 2\left| \frac{\nabla ^{N} (H-|\nu |^2_{\Pten })}{H-|\nu |^2_{\Pten }} + \Pten ^{TN}(\nu , \cdot ) \right| ^2 + |\mathring{\mathrm{II}}+ \mathring{\Pten} ^{TN}|^2\\
	& + \int_{N_t} \left( R - \frac{1}{2} (\Pten(\nu ,\nu  ) )^2 - H \Pten(\nu ,\nu ) - 2 |\Pten ^{TN}(\nu ,\cdot )|^2 + 2 (\mathrm{div}_{g}\Pten)(\nu ) + H \mathrm{tr}_g \Pten - |\mathring{\Pten} ^{TN}|^2 \right) .
\end{split}	
\end{align*} 
The equivalent DEC (\ref{eq:DEC2}) implies that
\begin{align*}
	R + \frac{1}{2}(\mathrm{tr}_g \Pten)^2  \geq  |\Pten|^2 + 2 |\mathrm{div}_{g} \Pten|.
\end{align*} 
Note that
\begin{align*}
	|\Pten|^2 = |\mathring{\Pten} ^{TN}|^2+ \frac{1}{2} (\mathrm{tr}_{N_t} \Pten ^{TN})^2 + 2|\Pten ^{TN}(\nu ,\cdot )|^2 + |\Pten(\nu ,\nu )|^2.
\end{align*} 
So
\begin{align*}
	&R - \frac{1}{2} (\Pten(\nu ,\nu  ) )^2 - H \Pten(\nu ,\nu ) - 2 |\Pten ^{TN}(\nu ,\cdot )|^2 + 2 (\mathrm{div}_{g}\Pten)(\nu ) + H \mathrm{tr}_g \Pten - |\mathring{\Pten} ^{TN}|^2 \\
	& \geq 16 \pi (\mu -|J|_g)
	+ \frac{1}{2} (\mathrm{tr}_{N_t} \Pten ^{TN})^2  - \frac{1}{2} (\mathrm{tr}_{g}\Pten)^2 + \frac{1}{2} (\Pten(\nu ,\nu  ) )^2 - H \Pten(\nu ,\nu ) + H \mathrm{tr}_g \Pten \\
	&\geq 16 \pi (\mu -|J|_g)+ \left( H- \Pten(\nu ,\nu ) \right) \left( \mathrm{tr}_g \Pten -  \Pten(\nu ,\nu ) \right)  \\
	&\geq 0.
\end{align*} 
That is,
\begin{align}\label{B'-ineq}
	\begin{split}
		16 \pi & e ^{-\frac{t}{2}}B'(t) \\
	&\geq  \int_{N_t} 2\left| \frac{\nabla ^{N} (H-|\nu |^2_{\Pten })}{H-|\nu |^2_{\Pten }} + \Pten ^{TN}(\nu , \cdot ) \right| ^2 + |\mathring{\mathrm{II}}+ \mathring{\Pten} ^{TN}|^2\\
	&\ \ + \int_{N_t}\left( 16 \pi (\mu -|J|_g)+ \left( H- \Pten(\nu ,\nu ) \right) \mathrm{tr}_{N_t} \Pten \right)   \\
	&\geq 0.
	\end{split}
\end{align} 

Thus, we have proved that if $(M^3,g ,k)$ is an initial data set satisfying the DEC and $\Pten \geq 0$, and if $(N_t)_{0 \leq t \leq T} \subset M^3$ is a smooth solution of the $\Pten$-IMCF so that each $N_t$ is closed and connected, then the following quantities 
\begin{align*}
	A(t),\ B(t),\ m_{\Pten}(N_t) := \sqrt{\frac{A(t)}{16 \pi }} B(t),
\end{align*} 
are all monotone increasing, whenever $A(t), B(t) \geq 0$.

Following the same arguments in \cite[Section 6]{Dong26SigmaIMCF}, together with the necessary modifications in Section \ref{section-preliminary}, we can prove the following monotonicity formula for weak solutions.

\begin{thm}\label{thm-weak-mono}
	Let $(M^3, g, \kten)$ be a complete, connected, asymptotically flat initial data set without boundary. Let $E_0$ be a precompact open set so that $\partial E_0$ is a connected $C ^{2, \alpha }$-smooth outermost past apparent horizon. Then there exists a weak solution $(N_t) _{0 \leq t< \infty}$ of the $\Pten$-IMCF with initial condition $E_0$ so that $m_{\Pten}(N_t)$ is monotone increasing, whenever the DEC (\ref{eq:DEC}) and the $2$-convexity condition $\mathbf{P} \geq 0$ hold.
\end{thm}

\begin{proof}
	The proof is the same as in \cite[Section 6]{Dong26SigmaIMCF} with necessary modifications.
Using the same notations as in \cite[Section 6]{Dong26SigmaIMCF}, we can replace all $\sigma _i$ by $\mathbf{P}$, and accordingly use the fact that when $t$ is not a jump time, $\tilde{N}_t ^{i} \to \tilde{N}_t$ in $C^{1,\alpha }$ and $|\tilde{\nu }_i|^2_{\mathbf{P}} \to |\tilde{\nu }|^2_{\mathbf{P}}$ in $C^{\alpha }$ on $\tilde{N}_t$. Then the only places where $\sigma ^{i} = |h| g$ has been used is the convergence of the term $\phi (\tilde{\nabla }_{\tilde{\nu }_i} \sigma ^{i})(\tilde{\nu }_i, \tilde{\nu }_i)$ and thus the arguments towards the final monotonicity formula. We indicate necessary modifications as follows.

	Recall that we have the same identity of \cite[Equation (30)]{Dong26SigmaIMCF}:
\begin{align}\label{ddt-identity-i}
	\begin{split}
		&\frac{d}{d t} \int_{\tilde{N}^{i}_t} \phi (\tilde{H}_{i}- |\tilde{\nu}_{i} |^2_{\mathbf{P} })^2 \\							&= \int_{\tilde{N}^{i}_t} (\tilde{H}_{i}-|\tilde{\nu}_i |^2_{ \mathbf{P}}) \tilde{\nabla} _{\tilde{\nu}_i }\phi -2 \langle \tilde{\nabla} ^{N} \phi , \frac{\tilde{\nabla} ^{N}(\tilde{H}_i-|\tilde{\nu}_i |^2_{\mathbf{P}})}{\tilde{H}_i- |\tilde{\nu}_i |^2_{\mathbf{P} }} \rangle \\
		&\ \ - 2 \int_{\tilde{N}^{i}_t}\phi \left(   \frac{|\tilde{\nabla} ^{N} (\tilde{H}_i-|\tilde{\nu}_i |^2_{ \mathbf{P}})|^2}{(\tilde{H}_i-|\tilde{\nu}_i |^2_{\mathbf{P} })^2} + \left( \tilde{\mathrm{Ric}}(\tilde{\nu}_i ,\tilde{\nu}_i ) +|\tilde{\mathrm{II}}|^2 \right) \right) \\
		&\ \ -2 \int_{\tilde{N}^{i}_t} \phi (\tilde{\nabla} _{\tilde{\nu}_i } \mathbf{P}) (\tilde{\nu}_i , \tilde{\nu}_i ) - 4 \int_{\tilde{N}^i_t} \phi \frac{\mathbf{P}( \tilde{\nu}_i , \tilde{\nabla} ^{N}(\tilde{H}_i- |\tilde{\nu}_i |^2_{\mathbf{P} }) )}{\tilde{H}_i- |\tilde{\nu}_i |^2_{\mathbf{P}}} + \int_{\tilde{N}^{i}_t}\phi \tilde{H}_i(\tilde{H}_i-|\tilde{\nu}_i |^2_{\mathbf{P} }).
	\end{split}
\end{align}

Similar to the identity (\ref{nabla-P-id}), we rewrite the term
\begin{align*}
	 (\tilde{\nabla }_{\tilde{\nu }_i} \mathbf{P})(\tilde{\nu }_i, \tilde{\nu }_i) = (\mathrm{div}_{\tilde{g}} \mathbf{P})(\tilde{\nu }_i) - \mathrm{div}_{\tilde{N}^{i}_t}(\mathbf{P}^{TN}(\tilde{\nu }_i, \cdot ) ) - H_{\tilde{N}^{i}_{t}} \mathbf{P}(\tilde{\nu }_i, \tilde{\nu }_i) + \langle \tilde{\mathrm{II}}, \mathbf{P}^{TN} \rangle.
\end{align*} 
Set $\tilde{\nu }_i ^{T} := \tilde{\nu }_i - \langle \tilde{\nu }_i, e_{z} \rangle e_{z}$. We already know that $\sup_{\tilde{N}^{i}_t \cap (M \times \mathrm{supp} \phi )}|\langle \tilde{\nu }_i, e_{z} \rangle| \to 0$ for a.e. $0\leq t \leq T$. So $\tilde{\nu }_i ^{T} \to \tilde{\nu } $ a.e. for a.e. $0\leq t \leq T$. Since $\tilde{g}$ and $\mathbf{P}$ are vertical invariant, we have 
\begin{align*}
	\phi (\mathrm{div}_{\tilde{g}} \mathbf{P})(\tilde{\nu }_i) = \phi (\mathrm{div}_{g} \mathbf{P})(\tilde{\nu }_i ^{T}).
\end{align*} 
By the bounded convergence theorem, for each $0 \leq r <s$, we have
\begin{align*}
	\int_{r}^{s} \int_{\tilde{N}^{i}_{t}} \phi (\mathrm{div}_{\tilde{g}} \mathbf{P})(\tilde{\nu }_i) \to \int_{r}^{s} \int_{\tilde{N}_{t}}\phi (\mathrm{div}_{g} \mathbf{P})(\tilde{\nu }).
\end{align*} 
Using integration by parts, similarly, since $\mathbf{P}(e_{z}, \cdot )=0$, we have
\begin{align*}
	2 \int_{r}^{s}\int_{\tilde{N}^{i}_{t}} \phi \mathrm{div}_{\tilde{N}^{i}_{t}}(\mathbf{P}^{TN}(\tilde{\nu }_i, \cdot ) ) &= 2 \int_{r}^{s}\int_{\tilde{N}^{i}_{t}} \phi ' \langle \tilde{\nu }_i, e_{z} \rangle \mathbf{P}(\tilde{\nu }_{i}, \tilde{\nu }_i ) \to 0.
\end{align*} 
The remaining terms can be controlled exactly as in \cite[Section 6]{Dong26SigmaIMCF}, together with the algebraic identities used to prove (\ref{B'-ineq}). 
\end{proof}

\begin{lem}
	Under the same assumption of Theorem \ref{thm-weak-mono}, and assume further that the DEC (\ref{eq:DEC}) and the $2$-convexity condition $\mathbf{P} \geq 0$ hold in a neighborhood of the end. Then 
	\begin{align*}
		\lim_{t\to \infty} m_{\mathbf{P}}(N_t) \leq m. 
	\end{align*} 
\end{lem}
\begin{proof}
	For $\lambda >0$, we introduce the rescaling
\begin{align*}
	\Omega ^{\lambda }&:= \lambda \Omega , g ^{\lambda }(y) := \lambda ^2 g(\lambda ^{-1} y), u ^{\lambda }(y):= u (\lambda ^{-1} y), E_{t}^{\lambda } := \lambda E_{t},\\
	\nu ^{\lambda }(y) &:= \nu (\lambda ^{-1} y), \mathbf{P}^{\lambda }(y) := \lambda ^{-1} \mathbf{P}(\lambda ^{-1} y), \mathbf{P}^{\lambda }(\nu ^{\lambda }, \nu ^{\lambda })(y) = \lambda ^{-1}\mathbf{P}(\nu ,\nu )(\lambda ^{-1} y).
\end{align*} 
Note that $0\leq \mathbf{P}(\nu ,\nu )(x) \leq \mathrm{tr}_{g} \mathbf{P} = 2 \mathrm{tr}_{g} \mathbf{k} = o(|x|^{-1})$ and $|\nabla \mathbf{P}| = o(|x| ^{-2})$ as $|x| \to \infty$. Using the Compactness Lemma \cite[Lemma 4.8]{Dong26SigmaIMCF}, we can argue as in the proof of \cite[Lemma 7.1]{Dong26SigmaIMCF} to conclude the same result of \cite[Lemma 7.1]{Dong26SigmaIMCF}. Finally, using the weak monotonicity formula of (\ref{B'-ineq}), we can follow the proof of \cite[Lemma 7.2]{Dong26SigmaIMCF} to conclude the proof.
\end{proof}

\begin{proof}[Proof of Theorem \ref{thm-main}]
	We can fill in the boundary $\Sigma $ by a $3$-ball to obtain a complete connected asymptotically flat $3$-manifold without boundary. Let $E_0$ be the region enclosed by $\Sigma $, and choose a weak solution of $\Pten$-IMCF starting from $E_0$. Then $E_0 = E_0 ^{+}$. The Connectedness Lemma and Theorem \ref{thm-weak-mono} therefore apply. Since $\Sigma $ is a past apparent horizon, $m_{\Pten}(\Sigma ) = \sqrt{\frac{|\Sigma |}{16 \pi }} $. The asymptotic comparison argument of \cite[Section 7]{Dong26SigmaIMCF} also applies. Thus
	\begin{align*}
		 \sqrt{\frac{|\Sigma |}{16 \pi }} = m_{\Pten}(\Sigma ) \leq \lim_{t\to \infty} m_{\Pten}(N_t) \leq m. 
	\end{align*} 

It remains to consider the equality case. For a.e. $t$, $H - \mathbf{P}(\nu ,\nu ) >0 $ a.e. on $N_{t}$, and by examining the equality conditions in the monotonicity formulas for both $A(t)$ and $B(t)$, we obtain that
\[
\Pten(\nu,\nu)=\mathrm{tr}_{N_t}\Pten=0 \text{ on } N_{t}.
\]
Hence $\Pten\equiv 0$, which also implies that $\mathbf{k} \equiv 0$ on $N_{t}$. So $\int_{N_t} |\nabla ^{N} (H_{N_t}- \mathbf{P}(\nu ,\nu ) )|^2 =0 $ for a.e. $t$. Thus $H_{N_t}- \mathbf{P}(\nu ,\nu ) = H_{N_t} = H(t)$ is a constant for a.e. $x \in N_t$. By standard elliptic theory, $N_{t}$ is smooth for a.e. $t$. Using the approximation $N_{t_i} \to N_{t}^{+}$ with $t_i \to t+$, and by lower semicontinuity and bounded convergence theorem, we conclude that $N_{t}^{+}$ is a smooth surface with constant mean curvature for all $t$. The reminder of the arguments proceeds as in \cite[Section 7]{Dong26SigmaIMCF}: if there is a jump time, i.e. $N_{t} \neq N_{t}^{+}$, then by the properties of weak solution, we know $H=\mathbf{P}(\nu ,\nu ) =0$ on a portion of $N_{t} ^{+} \setminus N_{t}$, which together with the fact that the mean curvature of $N_{t}^{+}$ is a constant implies that $N_{t}^{+}$ is a minimal surface. This contradicts with the outermost property of $\partial M$. Thus there is no jump time, which then implies that the weak $\mathbf{P}$-IMCF is a smooth solution of the standard IMCF. 
\end{proof}

\enlargethispage{3\baselineskip}
\bibliographystyle{alpha}
\bibliography{math}

\end{document}